\documentclass{amsart}
\usepackage{amssymb, amsmath, latexsym}
\usepackage[small,nohug,heads=vee]{diagrams}
\diagramstyle[labelstyle=\scriptstyle]

\newtheorem{thm}{Theorem}[section]
\newtheorem{lem}[thm]{Lemma}

\newtheorem{prop}[thm]{Proposition}

\theoremstyle{definition}

\theoremstyle{remark}
\newtheorem{rmk}[thm]{Remark}

\numberwithin{equation}{section}

\newcommand{\ord}{\text{ord}}

\newcommand{\z}{{\mathbb Z}}
\newcommand{\q}{{\mathbb Q}}

\newcommand{\rank}{\text{rank}}

\begin{document}

\title{Completely $p$-primitive  binary quadratic forms}

\author{Byeong-Kweon Oh and Hoseog Yu}

\address{Department of Mathematical Sciences and Research Institute of Mathematics, Seoul National University, Seoul 08826, Korea}
\email{bkoh@snu.ac.kr}
\thanks{This work of the first author  was supported by the National Research Foundation of Korea (NRF-2017R1A2B4003758).}

\address{Department of Applied Mathematics, Sejong University, Seoul 05006, Korea}
\email{hsyu@sejong.ac.kr}



\begin{abstract} Let $f(x,y)=ax^2+bxy+cy^2$ be a binary quadratic form with integer coefficients. For a prime $p$ not dividing the discriminant of $f$, we say $f$ is completely $p$-primitive if for any non-zero integer $N$,  the diophantine equation $f(x,y)=N$ has always an integer solution $(x,y)=(m,n)$ with  $(m,n,p)=1$ whenever it has an integer solution. In this article, we study various properties of completely $p$-primitive binary quadratic forms. In particular, we give a necessary and sufficient condition for a definite binary quadratic form $f$ to be completely $p$-primitive.
 %Some applications to the representations of ternary quadratic forms will be provided.   
\end{abstract}

\maketitle \pagestyle{myheadings}
 \markboth {Byeong-Kweon Oh}{Completely $p$-primitive binary quadratic forms}

\maketitle

\section{Introduction}  A two variable homogeneous quadratic  polynomial with integer coefficients
$$
f(x,y)=ax^2+bxy+cy^2 \ (a,b,c \in \z)
$$ 
is called a binary quadratic form if the discriminant $D=b^2-4ac$ is a non-square integer.  We always  assume that $f$ is primitive, that is, $(a,b,c)=1$, unless stated otherwise.  An integer $N$ is said to be (primitively) represented by $f$ if $f(x,y)=N$ has an integer solution $x,y$ ($x,y$ such that $(x,y)=1$, respectively). The set of all non-zero integers that are (primitively) represented by $f$ is denoted by $Q(f)$ ($Q^*(f)$, respectively). As one of diophantine equations, it is quite a difficult problem to decide $Q(f)$ for  an arbitrary binary quadratic form $f$.   

In 1928, in his unpublished thesis,  B. W. Jones proved that if an odd prime $p$ is represented by $x^2+ky^2$, where  $k$ is a positive integer relatively prime to $p$,  then the diophantine equation $x^2+ky^2=N \ (N>0)$ has always an integer solution $x,y$ with 
$(x,y,p)=1$,  if it has an integer solution.  This  lemma was used by many authors to solve some problems on representations of positive ternary quadratic forms (see, for example, \cite{j}, \cite{ka1}, \cite{ka2}, \cite{oh1} and \cite{s}).

The aim of this article is to fully generalize Jones' lemma stated above. To be more precise, let $f$ be a binary quadratic form with discriminant $D$ and let $p$ be a prime not dividing $D$. 
An integer $a$ is said to be {\it $p$-primitively represented} by $f$  if  $f(x,y)=a$ has an integer solution $(x,y)=(m,n)$ with $(m,n,p)=1$. The set of all integers that are $p$-primitively represented by $f$ is denoted by $Q_p^*(f)$.  Then, clearly
$$
Q^*(f) \subset Q_p^*(f) \subset Q(f)
$$
for any prime $p$. Note that $0 \not \in Q(f)$ from the definition.  A binary form $f$ is called {\it completely $p$-primitive} if $Q(f)=Q_p^*(f)$.  Jones' lemma stated above says that  $x^2+ky^2$ is completely $p$-primitive 
for any odd prime $p \in Q(x^2+ky^2)$ relatively prime to $k$.

Let $D$ be a non-square integer congruent to $0$ or $1$  modulo $4$.  Let $\mathfrak G_D$ be the set of all proper classes of primitive binary quadratic forms with discriminant $D$.  Then it is well known that $\mathfrak G_D$ forms an abelian group with the composition law (for details, see \cite{Cas}).  Let $\mathfrak I_D$ be the identity  class in $\mathfrak G_D$. For any proper class $\mathfrak A \in \mathfrak G_D$, the set $Q(\mathfrak A)$ denotes $Q(f)$ for any binary quadratic form $f \in \mathfrak A$.   Similarly, we also define  $Q^*(\mathfrak A)$ and $Q_p^*(\mathfrak A)$. For any two binary forms $f$ and $g$ having same discriminant, we write
$$
f \simeq g \iff \text{there is a proper class $\mathfrak C \in \mathfrak G_D$ such that $f,g \in \mathfrak C$.}
$$
A proper class $\mathfrak D \in \mathfrak G_D$ is called an ambiguous class if  $\mathfrak D^2=\mathfrak I_D$.  For a prime $p$, a proper class $\mathfrak A \in \mathfrak G_D$ is called {\it completely $p$-primitive} if any binary form $f \in \mathfrak A$  is completely $p$-primitive, or equivalently, $Q(\mathfrak A)=Q_p^*(\mathfrak A)$.

 In this article, we study various properties of completely $p$-primitive binary forms for any prime $p$. If $\left(\frac Dp\right)=-1$, where $\left(\frac{\cdot}{\cdot}\right)$ is the Kronecker's symbol, then one may easily check that no integers divisible by $p$ are $p$-primitively represented by $f$.  Hence we always assume that $\left(\frac Dp\right)=1$.   We prove that if  $p^2 \in Q_p^*(\mathfrak I_D)$, then any proper class in $\mathfrak G_D$ is completely $p$-primitive. Conversely, if an ambiguous class $\mathfrak D \in \mathfrak G_D$ is completely $p$-primitive, then $p^2 \in Q_p^*(\mathfrak I_D)$.  Finally, for any prime $p$ such that $p^2 
  \not \in Q_p^*(\mathfrak I_D)$, we prove that a proper class $\mathfrak A \in \mathfrak G_D$ is completely $p$-primitive if and only if the order of $\mathfrak A$ in the group $\mathfrak G_D$ is  $4$ and  $p^2 \in Q_p^*(\mathfrak A^2)$, under the assumption that $D<0$. 
  
  Some  basic notations and terminologies on binary quadratic forms, especially the composition law between binary forms having same discriminant,  can be found in \cite{Cas}. See also \cite{OM} for some  basic notations and terminologies on $\z$-lattices. For a binary quadratic form $f(x,y)=ax^2+bxy+cy^2$,  we simply write $f=[a,b,c]$.  

%%%%%%%%%%%%%%%%%%%%%%%%%%%%%%%%%%%%%%%%%%%%%%%%%%%%%%%%%%%%%%
%%%%%%%%%%%%%%%%%%%%%%%%%%%%%%%%%%%%%%%%%%%%%%%%%%%%%%%%%%%%%%
%%%%%%%%%%%%%%%%%%%%%%%%%%%%%%%%%%%%%%%%%%%%%%%%%%%%%%%%%%%%%%
%%%%%%%%%%%%%%%%%%%%%%%%%%%%%%%%%%%%%%%%%%%%%%%%%%%%%%%%%%%%%%
%%%%%%%%%%%%%%%%%%%%%%%%%%%%%%%%%%%%%%%%%%%%%%%%%%%%%%%%%%%%%%
\section{$p$-primitive representations of integers by binary quadratic forms}
%%%%%%%%%%%%%%%%%%%%%%%%%%%%%%%%%%%%%%%%%%%%%%%%%%%%%%%%%%%%%%
%%%%%%%%%%%%%%%%%%%%%%%%%%%%%%%%%%%%%%%%%%%%%%%%%%%%%%%%%%%%%%
%%%%%%%%%%%%%%%%%%%%%%%%%%%%%%%%%%%%%%%%%%%%%%%%%%%%%%%%%%%%%%
%%%%%%%%%%%%%%%%%%%%%%%%%%%%%%%%%%%%%%%%%%%%%%%%%%%%%%%%%%%%%%
%%%%%%%%%%%%%%%%%%%%%%%%%%%%%%%%%%%%%%%%%%%%%%%%%%%%%%%%%%%%%%

Let $D$ be a non-square integer such that $D \equiv 0$  or $1 \pmod 4$. Let $\mathfrak G_D$ be the abelian group  of proper equivalence classes of primitive binary quadratic forms with discriminant $D$.  Let $\mathfrak I_D$ be the identity class in $\mathfrak G_D$, that is, $h(x,y) \in \mathfrak I_D$, where 
$$
h(x,y)=\begin{cases} x^2-\frac D4 y^2  \quad &\text{if $D$ is even,}\\  
x^2+xy+\frac{1-D}4y^2  \quad &\text{otherwise}.\end{cases}
$$

\begin{lem} \label{down} A binary quadratic form $f$ is completely $p$-primitive if and only if 
$ap^2 \in Q_p^*(f)$ for any $a \in Q_p^*(f)$. 
\end {lem} 

\begin{proof}  Note that ``only if" part is trivial. Assume that $ap^2 \in Q_p^*(f)$ for any $a \in Q_p^*(f)$. For any $t \in Q(f)$, let $t=f(p^{\alpha}a,p^{\beta}b)$, where $a,b$ are integers  that  are either zero or not divisible by $p$. Then $\frac t{p^{2\min(\alpha,\beta)}} \in Q_p^*(f)$. Now, by assumption, we have  $t \in Q_p^*(f)$.  Therefore we have $Q(f)=Q_p^*(f)$.
\end{proof}

Let $\mathfrak A, \mathfrak Z$ be proper classes of primitive binary quadratic forms with discriminant $D$ and $f \in \mathfrak A$ and $g \in \mathfrak Z$.
 It is well known that for any integers $a \in Q(\mathfrak A)$ and $\alpha \in  Q(\mathfrak Z)$, $a\alpha \in Q(\mathfrak A\cdot\mathfrak Z)$ (see \cite{ef}).  Furthermore, if  $a \in Q^*(\mathfrak A), \alpha \in  Q^*(\mathfrak Z)$ and $(a,\alpha)=1$, then $a\alpha \in Q^*(\mathfrak A\cdot\mathfrak Z)$.
However, if $a$ and $\alpha$ are not relatively prime, then it is no longer true. For  example,  if $[2,1,4] \in  \mathfrak A$ and $[2,-1,4] \in \mathfrak Z$, then  $7 \in Q^*(\mathfrak A) \cap Q^*(\mathfrak Z)$. However,  $7\!\cdot \!7 \not \in Q^*(\mathfrak A\cdot\mathfrak Z)=Q^*([1,1,8])$. Note that  for any positive integer $a$, $a \in Q^*(\mathfrak A)$ if and only if there are integers $b,c$ such that $[a,b,c] \in \mathfrak A$.

\begin{lem} \label{keyl} Under the notations given above, let $f=[a,b,c] \in \mathfrak A$ and $g=[\alpha,\beta,\gamma] \in \mathfrak Z$. If $(a,\alpha,D)=1$ and $2(a,\alpha)$ divides $\beta-b$, then 
$a\alpha \in Q^*(\mathfrak A\cdot\mathfrak Z)$. 
\end{lem}

\begin{proof}  Let $a=da_1$ and $\alpha=d\alpha_1$, where $d=(a,\alpha)$.   Let $x_0, z_0$ be integers such that $a_1x_0-\alpha_1z_0=\frac {\beta-b}{2d}$. Note that for any integer $t$,
$$
[a,b,c] \simeq [a,2a(\alpha_1t+x_0)+b,f(\alpha_1t+x_0,1)]
$$
and 
$$
[\alpha,\beta,\gamma] \simeq [\alpha,2\alpha(a_1t+z_0)+\beta,g(a_1t+z_0,1)].
$$
Since $\det([a,b,c])=\det([\alpha,\beta,\gamma])=D$ and $2a(\alpha_1t+x_0)+b=2\alpha(a_1t+z_0)+\beta$, we have 
$$
af(\alpha_1t+x_0,1)=\alpha g(a_1t+z_0,1).
$$
 Therefore, $f(\alpha_1t+x_0,1)$ is divisible by $\alpha_1$. Furthermore, since
$$
ax_0^2+bx_0+c=\alpha_1a(\alpha_1 t^2+2tx_0)+(b\alpha_1)t-f(\alpha_1t+x_0,1)\equiv 0 \ \text{(mod $\alpha_1$)},
$$
 there is an integer $m$ such that $ax_0^2+bx_0+c=\alpha_1m$. Let $t_0$ be an integer such that $bt_0+m \equiv 0 \pmod d$. Note that this is possible, for $(b,d)=1$.  Since $f(\alpha_1t_0+x_0,1)=\alpha\eta$ for some integer $\eta$, we have
$$
[a,b,c][\alpha,\beta,\gamma] \simeq [a,2a(\alpha_1t_0+x_0)+b,\alpha\eta][\alpha,2\alpha(a_1t_0+z_0)+\beta,*] \simeq [a\alpha,*,*] \in \mathfrak A \cdot \mathfrak Z.
$$
Therefore, we have $a\alpha \in Q^*(\mathfrak A\cdot \mathfrak Z)$.  
\end{proof}

\begin{lem}  \label{tekl}  Under the notations given above, if $a \in Q_p^*(\mathfrak A)$, $\alpha \in Q(\mathfrak Z)$,  where $(a,\alpha)=(\alpha,p)=1$, then $a\alpha \in Q_p^*(\mathfrak A\cdot\mathfrak Z)$.\end{lem}

\begin{proof}  Assume that $a=f(dx_0,dy_0)$, where $(d,p)=(x_0,y_0)=1$. Then $\frac a{d^2} \in Q^*(\mathfrak A)$. Since $\alpha \in Q(\mathfrak Z)$ and $(\alpha,p)=1$, there is an integer $\delta$ relatively prime to $p$ such that $\frac \alpha{\delta^2} \in Q^*(\mathfrak Z)$. From the assumption that $(a,\alpha)=1$, we have $\frac {a\alpha}{d^2\delta^2} \in Q^*(\mathfrak A\cdot\mathfrak Z)$ by Lemma \ref{keyl}. Now, the lemma follows from the fact that $(p,d\delta)=1$. 
\end{proof}

\begin{prop} \label{keyp}   Under the notations given above, if $a \in Q_p^*(\mathfrak A)$, $\alpha \in Q_p^*(\mathfrak Z)$,  where $(a,\alpha,D)=1$, then $a\alpha \in Q_p^*(\mathfrak A\cdot\mathfrak Z) \cup Q_p^*(\mathfrak A\cdot\mathfrak Z^{-1})$.  \end{prop}

\begin{proof} Without loss of generality, we may assume that $a \in Q^*(\mathfrak A)$ and $\alpha \in Q^*(\mathfrak Z)$. 
Let $d=(a,\alpha)$. Let $f=[a,b,c] \in \mathfrak A$ and $g=[\alpha,\beta,\gamma] \in \mathfrak Z$, for some integers $b,c$ and $\beta,\gamma$. 
Since 
$D=b^2-4ac=\beta^2-4\alpha\gamma$, we have $\beta^2-b^2 \equiv 0 \pmod {4d}$. Furthermore, since $\left(\frac{\beta-b}2,\frac{\beta+b}2,d\right)=1$,  there are relatively prime integers $d_1, d_2$ such that $d=d_1d_2$, and
$$
\frac {\beta-b}2 \equiv 0 \pmod {d_1} \quad \text{and} \quad  \frac {\beta+b}2 \equiv 0 \pmod {d_2}.
$$
Choose integers $d_1', d_2', d_1''$ and $d_2''$ such that $a=d_1'd_2'\widetilde{a}$ and $\alpha=d_1''d_2''\widetilde{\alpha}$, where $(d_1,d_2'd_2''\cdot\widetilde{a}\widetilde{\alpha})=(d_2,d_1'd_1''\cdot\widetilde{a}\widetilde{\alpha})=1$. Note that $(d_1',d_1'')=d_1$ and $(d_2',d_2'')=d_2$.   
We define proper classes $\mathfrak A_{d_1'}$,  $\mathfrak A_{d_2'}$ and $\mathfrak A_{\widetilde{a}} $ satisfying
$$
[d_1',b,d_2'\widetilde{a}c] \in \mathfrak A_{d_1'} ,  \ \   [d_2',b,d_1'\widetilde{a}c] \in \mathfrak A_{d_2'}  \quad \text{and} \quad  [\widetilde{a},b,d_1'd_2'c] \in \mathfrak A_{\widetilde{a}} .
$$
Similarly, we also define proper classes $\mathfrak Z_{d_1''}, \mathfrak Z_{d_2''}$ and $\mathfrak Z_{\widetilde{\alpha}}$. Then, clearly we have 
$$
\mathfrak A=\mathfrak A_{d_1'}\cdot\mathfrak A_{d_2'}\cdot\mathfrak A_{\widetilde{a}} \quad \text{and} \quad \mathfrak Z=\mathfrak Z_{d_1''}\cdot\mathfrak Z_{d_2''}\cdot\mathfrak Z_{\widetilde{\alpha}}.
$$ 

If $p$ does not divide $d$, then the proof is almost trivial. Hence we assume that $p$ divides $d$.   
First, assume that $p$ divides $d_1$. Then by Lemma \ref{keyl}, we have
$d_1'd_1'' \in Q_p^*(\mathfrak A_{d_1'}\cdot\mathfrak Z_{d_1''})$.  Since 
$$
d_2'\widetilde{a}\cdot d_2''\widetilde{\alpha} \in Q(\mathfrak A_{d_2'}\cdot\mathfrak A_{\widetilde{a}}\cdot\mathfrak Z_{d_2''}\cdot\mathfrak Z_{\widetilde{\alpha}}) \quad \text{and}  \quad (d_1'd_1'',d_2'\widetilde{a}\cdot d_2''\widetilde{\alpha})=(d_2'\widetilde{a}\cdot d_2''\widetilde{\alpha},p)=1,
$$
  we have $a\alpha \in Q_p^*(\mathfrak A\cdot\mathfrak Z)$ by Lemma \ref{tekl}. Now, assume that $p$ divides $d_2$. In this case, we have
$d_2'd_2'' \in Q_p^*(\mathfrak A_{d_2'}\cdot\mathfrak Z_{d_2''}^{-1})$ by Lemma \ref{keyl}.  Since $d_1'\widetilde{a}\cdot d_1''\widetilde{\alpha} \in Q(\mathfrak A_{d_1'}\cdot\mathfrak A_{\widetilde{a}}\cdot\mathfrak Z_{d_1''}^{-1}\cdot\mathfrak Z_{\widetilde{\alpha}}^{-1})$, we may conclude that $a\alpha \in Q_p^*(\mathfrak A\cdot\mathfrak Z^{-1})$ by Lemma \ref{tekl}. 
\end{proof}

%%%%%%%%%%%%%%%%%%%%%%%%%%%%%%%%%%%%%%%%%%%%%%%%%%%%%%%%%%%%%%
%%%%%%%%%%%%%%%%%%%%%%%%%%%%%%%%%%%%%%%%%%%%%%%%%%%%%%%%%%%%%%
%%%%%%%%%%%%%%%%%%%%%%%%%%%%%%%%%%%%%%%%%%%%%%%%%%%%%%%%%%%%%%
%%%%%%%%%%%%%%%%%%%%%%%%%%%%%%%%%%%%%%%%%%%%%%%%%%%%%%%%%%%%%%
%%%%%%%%%%%%%%%%%%%%%%%%%%%%%%%%%%%%%%%%%%%%%%%%%%%%%%%%%%%%%%
\section{$p$-primitive representations of binary $\z$-lattices}
%%%%%%%%%%%%%%%%%%%%%%%%%%%%%%%%%%%%%%%%%%%%%%%%%%%%%%%%%%%%%%
%%%%%%%%%%%%%%%%%%%%%%%%%%%%%%%%%%%%%%%%%%%%%%%%%%%%%%%%%%%%%%
%%%%%%%%%%%%%%%%%%%%%%%%%%%%%%%%%%%%%%%%%%%%%%%%%%%%%%%%%%%%%%
%%%%%%%%%%%%%%%%%%%%%%%%%%%%%%%%%%%%%%%%%%%%%%%%%%%%%%%%%%%%%%
%%%%%%%%%%%%%%%%%%%%%%%%%%%%%%%%%%%%%%%%%%%%%%%%%%%%%%%%%%%%%%

Let $L$ and $M$ be non-classic integral $\z$-lattices in a quadratic space $V$. Here a $\z$-lattice is non-classic integral if its norm 
ideal $\mathfrak n(L)$ is $\z$. Let $p$ be a prime.  A representation $\sigma : M \to L$ is called {\it  $p$-primitive}  if $\sigma(M) \not \subset pL$.  We say $M$ is $p$-primitively represented by $L$ if there is a $p$-primitive representation from $M$ to $L$. Note that if $\sigma(M_p)$ is a primitive submodule of $L_p$, that is, $\sigma : M_p \to L_p$ is a primitive representation, then $\sigma$ is $p$-primitive. If $\rank(M)=1$, then the converse is also true. However, for the higher rank case, the converse is not true in general.  A $p$-primitive representation $\sigma : pL \to L$ is called {\it essential} if $\sigma \in O(\q\otimes L)$ has an infinite order.  Recall that an isometry $\sigma \in O(V)$ is called {\it proper} if $\det(\sigma)=1$, and {\it improper} otherwise. 

Let $D$ be a non-square integer congruent to $0$ or $1$ modulo $4$. 
Let $\mathfrak A$ be a proper class with discriminant $D$ and let $f=[a,b,c] \in \mathfrak A$ be a binary quadratic form.  
The binary $\z$-lattice  corresponding to $f$ (or $\mathfrak A$) is defined by $L_f=\z \bold x+\z \bold y$  such that 
$$
Q(\bold x)=a, \  \ 2B(\bold x,\bold y)=b \quad \text{and}  \quad Q(\bold y)=c.
$$   
Note that the binary lattice corresponding to $\mathfrak A$ is isometric to the binary lattice corresponding to $\mathfrak A^{-1}$.      
Conversely, for a binary $\z$-lattice $L=\z \bold x+\z \bold y$, the binary quadratic form $f_L$ corresponding to $L$ is defined by $f_L=[Q(\bold x),2B(\bold x,\bold y),Q(\bold y)]$. 
Note that the discriminant $D_L$ of the binary $\z$-lattice $L$ is defined by $D_L=Q(\bold x)Q(\bold y)-B(\bold x,\bold y)^2$, as usual.  

\begin{thm} \label{ppre} Let $D$ be a non-square integer congruent to $0$ or $1$ modulo $4$  and let $p$ be a prime such that $\left(\frac Dp\right)=1$. Then the followings are all equivalent:
\begin{itemize}
\item [(i)] $4p^2 \in Q_p^*([1,0,-D])$;
\item [(ii)] $p^2\in Q_p^*(\mathfrak I_D)$;
\item [(iii)] there is a binary $\z$-lattice $L=\z \bold x+\z \bold y$ with discriminant $-D/4$ and a proper $p$-primitive representation $\sigma  : pL \to L$;
\item [(iv)] for any binary $\z$-lattice $L=\z \bold x+\z \bold y$ with discriminant $-D/4$, there is a proper $p$-primitive representation $\sigma  : pL \to L$;
\item [(v)] there is a binary $\z$-lattice $L=\z \bold x+\z \bold y$ with discriminant $-D/4$ and a proper $p$-primitive essential representation $\sigma  : pL \to L$;
\item [(vi)] for any binary $\z$-lattice $L=\z \bold x+\z \bold y$ with discriminant $-D/4$, there is a proper $p$-primitive essential representation $\sigma  : pL \to L$.
\end{itemize}
\end{thm}

\begin{proof} Note that (i) $\iff$ (ii) is trivial. 
We will prove that
$$
\begin{diagram} 
{\rm (i)} &\rImplies &{\rm (vi)}             &\rImplies   &{\rm (iv)}               &                 &   \\ 
    &                &\dImplies  &                 &\dImplies      &                 &  \\
    &                &{\rm (v)}           &\rImplies    &{\rm(iii)}               &\rImplies &{\rm (i)}. \\
\end{diagram}
$$
Since proofs of all arrows placed in the middle are trivial, it suffices to show that  (i) $\Rightarrow$ (vi) and (iii) $\Rightarrow$ (i).   
 \vskip 0.2cm
 
\noindent {\bf (i) $\Rightarrow$ (vi).}   Let $m,n$ be integers such that $4p^2=m^2+(-D)n^2$ with $(m,n,p)=1$. Let $L=\z \bold x+\z \bold y$ be any $\z$-lattice  such that $f_L=[a,b,c]$, where $b^2-4ac=D$.   
Define $\sigma  : pL \to L$ such that 
$$
\sigma(p\bold x)=\frac{(m+bn)}2\bold x+(-an)\bold y \qquad \text{and} \qquad \sigma(p\bold y)=cn\bold x+\frac{(m-bn)}2\bold y.
$$
Then, one may easily check that $\sigma$ is, in fact,  a proper isometry of $\q \otimes L$ such that $\sigma(pL) \not \subset pL$. The characteristic polynomial $f_{\sigma}(x)$ of $\sigma$ is $x^2-\frac {m}px+1$. Since all roots of $f_{\sigma}$ are not roots of unity,  the order of $\sigma$  is not finite.  
 
 \vskip 0.2cm 
\noindent {\bf (iii) $\Rightarrow$ (i).} Assume that the corresponding binary quadratic form to $L$ is $f_L=[a,b,c]$, where $b^2-4ac=D$. Since $D$ is a non-square integer, we have $ac \ne 0$.  Assume that $\sigma  : pL \to L$ is a proper $p$-primitive representation  such that 
$$
\sigma(p\bold x)=u\bold x+v\bold y \quad \text{and} \quad \sigma(p\bold y)=r\bold x+s\bold y.
$$
Then we have
\begin{align} 
\label{11} &p^2a=au^2+buv+cv^2\!,\\
\label{22}&p^2c=ar^2+brs+cs^2\!,\\
\label{33}&p^2b=2ura+(us+vr)b+2vsc,\\
\label{44}&us-vr= p^2.
\end{align}
By (\ref{33}) and (\ref{44}), we have $(au+bv)r+cvs=0$. Since at least one of $au+bv$ and $cv$ is non-zero, 
 there is a rational number $\mu$ such that
\begin{equation} \label{111}
r=-cv\mu \qquad \text{and}  \qquad s=(au+bv)\mu. 
\end{equation}
 Now, by (\ref{22}) and (\ref{111}),  $a \mu=\pm 1$.   Since $a$ divides both $cv$ and $bv$ and $(a,(b,c))=1$, $a$ divides $v$.  By letting $v=at$ in  (\ref{11}) for some integer $t$, we have 
 \begin{equation} \label{result}
4p^2=(2u+bt)^2+(-D)t^2.
\end{equation}
 Assume that $p$ divides $t$. Then $p$ divides $u$ by (\ref{result}), and  $p$ also divides both $r$ and $s$ by (\ref{111}). This is a contradiction to the assumption that $\sigma$ is $p$-primitive. Therefore, $p$ does not divide $t$ and hence $4p^2 \in Q_p^*([1,0,-D])$.   This completes the proof. 
 \end{proof}
 
 The following lemma will be used in the proof of the main theorem. 
 
 \begin{lem}  \label{count}  Let $L$ be a binary $\z$-lattice with $dL=-D/4$ whose isometry group $O(L)$ contains an improper isometry $\sigma$.  Then, for any vector $\bold x \in L$,  
 $$
 \ord_q(Q(\bold x-\sigma(\bold x)))\equiv \ord_q(Q(\bold x+\sigma(\bold x))) \equiv 0 \ (\text{mod $2$}), 
 $$ 
 for any prime $q$ not dividing $D$.  In particular,  if $\bold x \not \in pL$ and $Q(\bold x) \equiv 0 \ (\text{mod } p)$, where $p$ is a prime not dividing $D$, then $\sigma(\bold x) \ne \pm \bold x$.  
\end{lem}

\begin{proof}  Since $\sigma$ is an improper isometry, there is a primitive vector $\bold z \in L$ such that $\sigma=\tau_{\bold z}$ by Theorem 43.3 of \cite{OM}. Let $L=\z \bold z+\z \bold w$. Since $\tau_{\bold z}(\bold w)=\bold w-\frac{2B(\bold z,\bold w)}{Q(\bold z)} \bold z \in L$, there is an integer $t$ such that $2B(\bold z,\bold w)=Q(\bold z)t$. If $\bold x=a\bold z+b\bold w$, then $\sigma(\bold x)=-a\bold z+b(\bold w-t\bold z)$. Hence 
$$
Q(\bold x-\sigma(\bold x))=(2a+bt)^2Q(\bold z) \quad \text{and} \quad  Q(\bold x+\sigma(\bold x))=b^2(4Q(\bold w)-Q(\bold z)t^2).
$$ 
The lemma follows directly from the fact that $D=Q(\bold z)(Q(\bold z)t^2-4Q(\bold w))$. 

Now, assume that $\bold x \in L-pL$ such that $Q(\bold x) \equiv 0 \ (\text{mod } p)$. If $\sigma(\bold x)=-\bold x$, then $b=0$ and hence $\bold x=a\bold z$.  This is a contradiction. Suppose that $\sigma(\bold x)=\bold x$. Then $2a=-bt$ and 
$$
Q(\bold x)=a^2Q(\bold z)+2abB(\bold z,\bold w)+b^2Q(\bold w)=\frac{b^2}4(4Q(\bold w)-Q(\bold z)t^2).
$$ 
Since $D$ is not divisible by $p$ by assumption, both $b$ and $a$ are divisible by $p$, which is also a contradiction. 
\end{proof}

%%%%%%%%%%%%%%%%%%%%%%%%%%%%%%%%%%%%%%%%%%%%%%%%%%%%%%%%%%%%%%
%%%%%%%%%%%%%%%%%%%%%%%%%%%%%%%%%%%%%%%%%%%%%%%%%%%%%%%%%%%%%%
%%%%%%%%%%%%%%%%%%%%%%%%%%%%%%%%%%%%%%%%%%%%%%%%%%%%%%%%%%%%%%
%%%%%%%%%%%%%%%%%%%%%%%%%%%%%%%%%%%%%%%%%%%%%%%%%%%%%%%%%%%%%%
%%%%%%%%%%%%%%%%%%%%%%%%%%%%%%%%%%%%%%%%%%%%%%%%%%%%%%%%%%%%%%
\section{Classification of completely $p$-primitive binary quadratic forms}
%%%%%%%%%%%%%%%%%%%%%%%%%%%%%%%%%%%%%%%%%%%%%%%%%%%%%%%%%%%%%%
%%%%%%%%%%%%%%%%%%%%%%%%%%%%%%%%%%%%%%%%%%%%%%%%%%%%%%%%%%%%%%
%%%%%%%%%%%%%%%%%%%%%%%%%%%%%%%%%%%%%%%%%%%%%%%%%%%%%%%%%%%%%%
%%%%%%%%%%%%%%%%%%%%%%%%%%%%%%%%%%%%%%%%%%%%%%%%%%%%%%%%%%%%%%
%%%%%%%%%%%%%%%%%%%%%%%%%%%%%%%%%%%%%%%%%%%%%%%%%%%%%%%%%%%%%%

Let $D$ be a non-square integer congruent to $0$ or $1$ modulo $4$, and let $\mathfrak G_D$ be a group of proper classes of primitive binary quadratic forms with discriminant $D$. Let $p$ be a prime satisfying $\left( \frac Dp\right)=1$. 

\begin{thm} \label{ambithm} If $p^2 \in Q_p^*(\mathfrak I_D)$, where $\mathfrak I_D \in \mathfrak G_D$ is the identity class, then every proper class in $\mathfrak G_D$ is completely $p$-primitive. Conversely,  if  an ambiguous class in $\mathfrak G_D$  is completely $p$-primitive, then $p^2 \in Q_p^*(\mathfrak I_D)$. 
 \end{thm}

\begin{proof}  First,  assume that $p^2 \in Q_p^*(\mathfrak I_D)$. Let $\mathfrak A$ be any proper class in $\mathfrak G_D$.  Assume that $a \in Q_p^*(\mathfrak A)$.  Then by Proposition \ref{keyp}, we have
$$
ap^2 \in Q_p^*(\mathfrak A\cdot \mathfrak I_D) \cup Q_p^*(\mathfrak A\cdot \mathfrak I_D^{-1})=Q_p^*(\mathfrak A). 
$$
Therefore $\mathfrak A$ is completely $p$-primitive by Lemma \ref{down}.

Conversely, assume that an ambiguous class $\mathfrak B \in \mathfrak G_D$ is completely $p$-primitive. Let $L=\z \bold x+\z \bold y$ be a binary $\z$-lattice such that $f_L=[a,b,c] \in \mathfrak B$.  Since $L$ represents at least one prime primitively by \cite{w}, we may assume that $a$ is a prime. Since $p^2a \in Q(L)=Q_p^*(L)$, there are integers  $u,v$ such that 
$$
p^2a=Q(u\bold x+v\bold y)=au^2+buv+cv^2 \quad \text{and} \quad  (p,u,v)=1.
$$ 
Since $a$  divides $(bu+cv)v$, $a$ divides $v$ or $bu+cv$.  First, assume that $a$ divides $v$. If $v=at$ for some integer $t$, then  $4p^2=(2u+bt)^2+(-D)t^2$. One may easily check that $t$ is not divisible by $p$ from the assumption that $(p,u,v)=1$.  
Therefore, we have $p^2 \in Q_p^*(\mathfrak I_D)$ by Theorem \ref{ppre}. 
Finally, assume that $a$ divides $bu+cv$. Define $\sigma : pL \to L$ such that 
$$
\sigma(p\bold x)=u\bold x+v\bold y \quad \text{and} \quad \sigma(p\bold y)=\left(\frac{bu+cv}a\right)\bold x-u\bold y.
$$
Then, one may easily show that $\sigma$ is an improper $p$-primitive representation. Since $\mathfrak B$ is an ambiguous class and $f_L=[a,b,c] \in \mathfrak B$, there is an improper isometry $\tau \in O(L)$. Then $\tau\sigma : pL \to L$ is a proper $p$-primitive representation. The theorem follows from  Theorem \ref{ppre}.   \end{proof}

From now on, we always assume that $D<0$.  For a positive definite binary quadratic form $[a,b,c] \in \mathfrak C \in \mathfrak G_D$ and a positive integer $n$, we define
$$
r(n,[a,b,c])=r(n,\mathfrak C)=\vert\{ (x,y) \in \z^2 : ax^2+bxy+cy^2=n \}\vert,
$$
and
$$
r_p^*(n,[a,b,c])=r_p^*(n,\mathfrak C)=\vert\{ (x,y) \in \z^2 : ax^2+bxy+cy^2=n, (x,y,p)=1 \}\vert.
$$
We also define $r_p^{\flat}(n,[a,b,c]):=r(n,[a,b,c])-r_p^*(n,[a,b,c])$. 
It is well known that 
$$
R(n,\mathfrak G_D):=\sum_{\mathfrak C \in \mathfrak G_D} r(n,\mathfrak C)=\omega \sum_{k \mid n} \left(\frac Dk\right),
$$
where $\left(\frac{\cdot}{\cdot}\right)$ is the Kronecker's symbol and 
$$
\omega=\begin{cases} 6 \quad &\text{if $D=-3$},\\
       4 \quad &\text{if $D=-4$},\\
   2    \quad &\text{otherwise}.\\  \end{cases}
$$
Similarly, we define $R_p^*(n,\mathfrak G_D):=\sum_{\mathfrak C \in \mathfrak G_d} r_p^*(n,\mathfrak C)$.

\begin{thm} \label{classification} Assume that $p^2 \not \in Q_p^*(\mathfrak I_D)$. A proper class $\mathfrak A \in \mathfrak G_D$ is completely $p$-primitive if and only if  $p^2 \in Q_p^*(\mathfrak A^2)$ and the order of $\mathfrak A$ in $\mathfrak G_D$ is $4$. 
\end{thm}

\begin{proof} Note that $p^2 \in Q_p^*([1,1,1])$ for any prime $p \equiv 1 \ (\text{mod } 3)$ and $q^2 \in Q_q^*([1,0,1])$ for any prime $q \equiv 1 \ (\text{mod } 4)$. Hence we may assume that $D<-4$ and the order of isometry group of $f$ is $2$  or $4$ for any binary quadratic form $f$ with discriminant $D$.  

First, we prove ``if" part. Assume that  $\mathfrak A^4=\mathfrak I_D$ in $\mathfrak G_D$  and $p^2 \in Q_p^*(\mathfrak A^2)$. 
For any integer $a \in Q_p^*(\mathfrak A)$,  we have
$$
ap^2 \in Q_p^*(\mathfrak A\cdot \mathfrak A^2) \cup Q_p^*(\mathfrak A\cdot \mathfrak A^{-2})=Q_p^*(\mathfrak A^{-1})= Q_p^*(\mathfrak A),
$$
by Proposition \ref{keyp}.  Therefore, $\mathfrak A$ is completely $p$-primitive by Lemma \ref{down}.

To prove ``only if" part, assume that $\mathfrak A$ is completely $p$-primitive. Note that $\mathfrak A$ is not an ambiguous class by theorem \ref{ambithm}.  Let $q \in Q_p^*(\mathfrak A)$ be an odd prime not dividing $p\cdot D$.  Let $\mathfrak D \in \mathfrak G_D$  be a non-identity proper class such that $p^2 \in Q_p^*(\mathfrak D)$. First, we will show that $p^2 \in Q_p^*(\mathfrak A^2)$ by proving that $\mathfrak D=\mathfrak A^2$ or $\mathfrak D=\mathfrak A^{-2}$.

 Assume that $\mathfrak D$ is not an ambiguous class. 
By Lemma \ref{tekl}, we have
$$
qp^2 \in Q_p^*(\mathfrak A\cdot\mathfrak D) \cap Q_p^*(\mathfrak A\cdot\mathfrak D^{-1}) \cap Q_p^*(\mathfrak A^{-1}\cdot\mathfrak D) \cap Q_p^*(\mathfrak A^{-1}\cdot\mathfrak D^{-1}).
$$
Since $\mathfrak A$ is not an ambiguous class, any of three classes among $\mathfrak A\cdot\mathfrak D, \  \mathfrak A^{-1}\cdot\mathfrak D^{-1},\   \mathfrak A^{-1}\cdot\mathfrak D$ and $\mathfrak A\cdot\mathfrak D^{-1}$ cannot be the same simultaneously.  Suppose that all four proper classes are different with each other  in $\mathfrak G_D$. 
 Note that $r_p^{\flat}(qp^2,\mathfrak A)=r_p^{\flat}(qp^2,\mathfrak A^{-1}) \ge 2$. Since $R(qp^2,\mathfrak G_D)=12$, we have
$$
\begin{array} {rl}
r_p^{\flat}(qp^2,\mathfrak A) =r_p^{\flat}(qp^2,\mathfrak A^{-1})\!\!\!&=r_p^*(qp^2,\mathfrak A\cdot\mathfrak D)=r_p^*(qp^2,\mathfrak A\cdot\mathfrak D^{-1})\\
&=r_p^*(qp^2,\mathfrak A^{-1}\cdot\mathfrak D)=r_p^*(qp^2,\mathfrak A^{-1}\cdot\mathfrak D^{-1})=2.
\end{array}
$$
Since $qp^2 \in Q_p^*(\mathfrak A)$ by Lemma \ref{down}, we have $\mathfrak D=\mathfrak A^2$ or $\mathfrak A^{-2}$.  Assume that $\mathfrak A\cdot\mathfrak D=\mathfrak A^{-1}\cdot\mathfrak D^{-1}$ and  $\mathfrak A\cdot\mathfrak D^{-1} \ne \mathfrak A^{-1}\cdot\mathfrak D$.  Then, the binary $\z$-lattice corresponding to a binary form in $\mathfrak A\cdot\mathfrak D$  has an improper isometry. Therefore we have $r_p^*(qp^2,\mathfrak A\cdot\mathfrak D) \ge 4$ by Lemma \ref{count}.  Furthermore, one may easily show that  $r_p^*(qp^2,\mathfrak A\cdot\mathfrak D)=4$ and 
$$
r_p^{\flat}(qp^2,\mathfrak A) =r_p^{\flat}(qp^2,\mathfrak A^{-1})=r_p^*(qp^2,\mathfrak A\cdot\mathfrak D^{-1})=r_p^*(qp^2,\mathfrak A^{-1}\cdot\mathfrak D)=2.
$$
Since $qp^2 \in Q_p^*(\mathfrak A)$, we have $\mathfrak D=\mathfrak A^2$. The proof of the case when  $\mathfrak A\cdot\mathfrak D \ne \mathfrak A^{-1}\cdot\mathfrak D^{-1}$ and  $\mathfrak A\cdot\mathfrak D^{-1} = \mathfrak A^{-1}\cdot\mathfrak D$ is quite similar to the above.   Finally, if $\mathfrak A\cdot\mathfrak D=\mathfrak A^{-1}\cdot\mathfrak D^{-1}$ and  $\mathfrak A\cdot\mathfrak D^{-1}=\mathfrak A^{-1}\cdot\mathfrak D$, then one may easily show that $\mathfrak D=\mathfrak I_D$, which is a contradiction.

Now, assume that $\mathfrak D$ is an ambiguous class. Choose integers $b, c$ suitably so that 
$$
[q,b,p^2c] \in \mathfrak A\quad \text{and} \quad [p^2,b,qc] \in \mathfrak D.
$$ 
Since $r_p^*(p^2,\mathfrak D) \ge 4$ by Lemma \ref{count}, there is a vector $(x,y) \in \z^2-\{\pm(1,0)\}$ such that 
$$
p^2x^2+bxy+qcy^2=p^2\quad \text{and} \quad (x,y,p)=1.
$$  
Note that $[qp^2,b,c] \in \mathfrak A\cdot\mathfrak D$. Since $qp^2x^2+bx(qy)+c(qy)^2=qp^2$, $(x,qy) \ne (1,0)$ and $(x,qy,p)=1$, we have $r_p^*(qp^2,\mathfrak A\cdot\mathfrak D) \ge 4$. Similarly,  $r_p^*(qp^2,\mathfrak A^{-1}\cdot\mathfrak D) \ge 4$. Therefore, we have 
$$
r_p^{\flat}(qp^2,\mathfrak A) =r_p^{\flat}(qp^2,\mathfrak A^{-1})=2 \quad \text{and} \quad r_p^*(qp^2,\mathfrak A\cdot\mathfrak D)=r_p^*(qp^2,\mathfrak A^{-1}\cdot\mathfrak D)=4.
$$
Since $qp^2 \in Q_p^*(\mathfrak A)$, we have $\mathfrak D=\mathfrak A^2$.  Therefore, $p^2 \in Q_p^*(\mathfrak A^2)$ in any cases.

Suppose that $\mathfrak A^4 \ne \mathfrak I_D$ and $\mathfrak A^6 \ne \mathfrak I_D$.   Note that $p^2 \in Q_p^*(\mathfrak A^2)$ and $p^4 \in Q_p^*(\mathfrak A^4)$ by Lemma \ref{keyl}.  Hence we have 
$$
qp^4 \in Q_p^*(\mathfrak A^5) \cap Q_p^*(\mathfrak A^{-3}) \cap Q_p^*(\mathfrak A^3) \cap Q_p^*(\mathfrak A^{-5}).
$$
 Note that $R_p^*(qp^4,\mathfrak G_D)=20-R(qp^2,\mathfrak G_D)=20-12=8$.  Assume that $\mathfrak A^{10} \ne \mathfrak I_D$  and $\mathfrak A^{8} \ne \mathfrak I_D$. 
Then by Lemma \ref{count},  we have 
$$
r_p^*(qp^4,\mathfrak A^5)=r_p^*(qp^4,\mathfrak A^{-3})=r_p^*(qp^4,\mathfrak A^{3})=r_p^*(qp^4,\mathfrak A^{-5})=2.
$$
This is a contradiction, for $qp^4 \in Q_p^*(\mathfrak A)$ by assumption.
Assume that $\mathfrak A^{10} = \mathfrak I_D$. Then, we have
$$
r_p^*(qp^4,\mathfrak A^5)=4 \quad \text{and} \quad r_p^*(qp^4,\mathfrak A^{-3})=r_p^*(qp^4,\mathfrak A^{3})=2.
$$
Since $qp^4 \in Q_p^*(\mathfrak A)$,  the order of $\mathfrak A$ is $4$, which is a contradiction.
Now, assume that $\mathfrak A^{8} = \mathfrak I_D$. 
Since $\mathfrak A^4$ is an ambiguous class, $r_p^*(qp^4,\mathfrak A^5)=r_p^*(qp^4,\mathfrak A^{-5})=4$ by a similar reasoning given above. Since $qp^4 \in Q_p^*(\mathfrak A)$, $\mathfrak A^4=\mathfrak I_D$ or $\mathfrak A^6=\mathfrak I_D$, which is a contradiction. 

In the remaining, we prove that  the order of $\mathfrak A$ is $4$ by showing $\mathfrak A^6 \ne \mathfrak I_D$.  Suppose, on the contrary, that $\mathfrak A^6=\mathfrak I_D$. 
Assume that $[p^2,b,c] \in \mathfrak A^2$. Let $\mathfrak F$ be the proper class in $\mathfrak G_D$ such that  $[p,b,pc] \in \mathfrak F$. 
Then clearly $\mathfrak F^2=\mathfrak A^2$.  Suppose that $\mathfrak F=\mathfrak A$ or $\mathfrak A^{-1}$, that is, $p$ is represented by $\mathfrak A$. Since $\mathfrak A^3=\mathfrak A^{-3}$ by assmption, we have
$$
r_p^{\flat}(p^3,\mathfrak A)=r_p^{\flat}(p^3,\mathfrak A^{-1})=2 \quad\text{and} \quad r(p^3,\mathfrak A^3)=4.
$$
Since $R(p^3,\mathfrak G_D)=8$, we have  $p^3 \not \in Q_p^*(\mathfrak A)$, which is a contradiction by Lemma \ref{down}. From now on, we assume that $\mathfrak F \ne \mathfrak A$ and $\mathfrak F \ne \mathfrak A^{-1}$. Choose an odd prime $r \in Q_p^*(\mathfrak F^{-1}\cdot\mathfrak A)$ not dividing  $p\cdot D$. Assume that $pr \in Q_p^*(\mathfrak X)$ for some $\mathfrak X \in \mathfrak G_D$. Then by Lemma 2.2 of \cite{ef}, we have
$$
p \in Q_p^*(\mathfrak X\cdot\mathfrak F^{-1}\cdot\mathfrak A) \cup Q_p^*(\mathfrak X^{-1}\cdot\mathfrak F^{-1}\cdot\mathfrak A). 
$$ 
Since $\mathfrak F$ and $\mathfrak F^{-1}$ are the only proper classes containing  a binary quadratic form representing $p$, we have $\mathfrak X=\mathfrak A$ or $\mathfrak A^{-1}$. Furthermore, since $R(pr, \mathfrak G_D)=8$, we have 
$$
r_p^*(pr,\mathfrak A)=r_p^*(pr,\mathfrak A^{-1})=4.
$$ 
Now, let $f=[r,B,p^3C]\in \mathfrak F^{-1}\cdot\mathfrak A$ and $g=[p^3,B,rC] \in  \mathfrak F^3$ be binary quadratic forms. Note that 
$h=[p^3r,B,C] \in \mathfrak F^{-1}\cdot\mathfrak A\cdot \mathfrak F^3=\mathfrak A^3$. 
Since $\mathfrak F^3=\mathfrak F^{-3}$,  there is  a vector $(u,v) \in \z^2-\{\pm(1,0)\}$ such that 
$$
g(u,v)=p^3u^2+Buv+rCv^2=p^3, \  \ (u,v,p)=1.
$$
Note that $\ord_r(g(u\pm1,v))=\ord_r(p^3(2\pm2u)\pm Bv) \equiv 0 \pmod 2$ by Lemma \ref{count}. Since 
$$
h(1,0)=h(u,rv)=p^3ru^2+Bu(rv)+C(rv)^2=p^3r \quad \text{and} \quad (u,rv,p)=1,
$$
the binary $\z$-lattice $L_h=\z \bold x+\z \bold y$ corresponding to $h$ has at least $4$ primitive vectors $\pm \bold x$ and 
$\pm(u\bold x+rv\bold y)$,
all of whose norms are $p^3r$. Since the proper class $\mathfrak A^3$ containing $h$ is an ambiguous class, there is an improper isometry $\sigma \in O(L_h)$.  Since
$$
\ord_r(Q(u\bold x+rv\bold y \pm \bold x))=\ord_r(h(u\pm1,rv))=1+\ord_r(p^3(2\pm 2u)\pm Bv) \equiv 1 \ (\text{mod } 2),
$$
$\sigma(u\bold x+rv\bold y)\ne \pm \bold x$ by Lemma \ref{count}. Therefore 
$$
\pm \bold x, \  \ \pm (u\bold x+rv\bold y), \  \ \pm \sigma(\bold x) \ \ \text{and}  \ \  \pm \sigma(u\bold x+rv\bold y)
$$
 are all different primitive vectors  in $L_h$ whose norms are $p^3r$.  Furthermore, since 
 $$
 r_p^{\flat}(p^3r,\mathfrak A)=r_p^{\flat}(p^3r,\mathfrak A^{-1})=4 \quad \text{and}  \quad R(p^3r,\mathfrak G_D)=16,
 $$
  we have  $r_p^*(p^3r,\mathfrak A^3)=8$.   This implies that $p^3r$ is not primitively represented by $\mathfrak A$, though $pr$ is primitively represented by $\mathfrak A$. This is a contradiction. Therefore the order of $\mathfrak A$ in $\mathfrak G_D$ is $4$. This completes the proof.
\end{proof}

\begin{rmk} {\rm (i) Note that $\mathfrak G_{-56} \simeq \z/4\z$ and the proper class $\mathfrak A \in \mathfrak G_{-56}$ containing $[3,2,5]$ is of order $4$. Furthermore, 
 $36 \in Q_3^*(\mathfrak A^2)$, where $[2,0,7] \in \mathfrak A^2$. Therefore, for any positive integer $a$ such that the diophantine equation $3y^2+2yz+5z^2=a$ has an integer solution, it has an integer solution $(y,z)=(b,c)$ such that $(b,c,3)=1$ by Theorem \ref{classification}.
 \vskip 0.2cm 
 
\noindent (ii)  For a positive integer $m \not \equiv 0 \pmod 3$, let $f_m(x,y,z)=m^2x^2+3y^2+2yz+5z^2$ be a ternary quadratic form.  Note that $f_m(x,y,z) \equiv 1 \pmod 3$ if and only if 
 $$
 \begin{array} {rl}
 (x,y,z) \equiv (0,\pm1,\pm1), \!\!\! &(\pm1,-1,0), \ (\pm1,1,-1), \\ 
                                 &(\pm1, -1,1), \ (\pm1,1,0), \ (\pm1,0,0) \pmod 3.
 \end{array}
 $$
Assume that for an integer $N \equiv 1 \pmod 3$, $f_m(x,y,z)=N$ has an integer solution. Then there is an integer solution
$f_m(a,b,c)=N$ with $a-b+c \equiv 0 \pmod 3$, except the case when $N=f(x,0,0)=m^2x^2$ by Theorem \ref{classification}. Therefore, if we define $\widetilde{f_m}(x,y,z)=m^2(3x+y-z)^2+3y^2+2yz+5z^2$, then
$$
Q(f_m) \cap  \{ k \in S_{3,1} : k \not \in m^2\z^2 \} =Q(\widetilde{f_m})) \cap  \{ k \in S_{3,1} : k \not \in m^2\z^2\},
$$
 where $S_{3,1}$ denotes the set of positive integers that are congruent to $1$ modulo $3$.
In particular, assume that $m=1$. Then one may easily show that $\widetilde{f_1}=4x^2+6y^2+7z^2+6yz+2zx$. In fact, the class number of  $\widetilde{f_1}$ is three and the other two ternary quadratic forms in the genus of $\widetilde{f_1}$ are 
$$
g(x,y,z)=x^2+3y^2+42z^2\quad  \text{and} \quad h(x,y,z)=x^2+6y^2+21z^2.
$$
 One may easily check that $d\widetilde{f_1}=2\cdot3^2\cdot7$, and $2^2, 3^2$, $7^2$  are all represented by $\widetilde{f_1}$.  For any prime $p$ not dividing $2\cdot3\cdot7$, the isometry class $[\widetilde{f_1}]$ is connected to $[g]$ or $[h]$ in the graph $\mathfrak G_{f,p}(0)$ defined in \cite{jlo} (see also \cite{bh}). Furthermore, since $1 \in Q(g)\cap Q(h)$, $p^2$ is represented by $\widetilde{f_1}$. Therefore, every square of an integer except $1$ is represented by $\widetilde{f_1}$. Consequently, we have 
$$
Q(f_1) \cap  S_{3,1} -\{1\}=Q(\widetilde{f_1}) \cap S_{3,1}.
$$ } 
\end{rmk}

\end{document}